\documentclass[a4paper,12pt,reqno]{amsart}

\usepackage{amsmath}
\usepackage{amsfonts}
\usepackage{amssymb}
\usepackage{amsthm}

\numberwithin{equation}{section}
\newtheorem{theorem}[equation]{Theorem}
\newtheorem{lemma}[equation]{Lemma}
\newtheorem{proposition}[equation]{Proposition}
\newtheorem{corollary}[equation]{Corollary}

\theoremstyle{definition}
\newtheorem{definition}[equation]{Definition}

\theoremstyle{definition}
\newtheorem{remark}[equation]{Remark}
\def\kint_#1{\mathchoice%
          {\mathop{\kern 0.2em\vrule width 0.6em height 0.69678ex depth -0.58065ex
                  \kern -0.8em \intop}\nolimits_{\kern -0.4em#1}}%
          {\mathop{\kern 0.1em\vrule width 0.5em height 0.69678ex depth -0.60387ex
                  \kern -0.6em \intop}\nolimits_{#1}}%
          {\mathop{\kern 0.1em\vrule width 0.5em height 0.69678ex depth -0.60387ex
                  \kern -0.6em \intop}\nolimits_{#1}}%
          {\mathop{\kern 0.1em\vrule width 0.5em height 0.69678ex depth -0.60387ex
                  \kern -0.6em \intop}\nolimits_{#1}}}

\newcommand{\R}{\mathbb{R}}

\newcommand{\N}{\mathbb{N}}

\providecommand{\ch}[1]{\text{\raise 2pt \hbox{$\chi$}\kern-0.2pt}_{#1}}

\newcommand\diam{\operatorname{diam}}

\begin{document}
\title{Smoothness spaces of higher order \\
on lower dimensional subsets of \\
the Euclidean space}
\author{Lizaveta Ihnatsyeva AND Riikka Korte}
\date{\today}
\begin{abstract}
%We study the relation between 

%We provide an extension of Shvartsman's Calder\'on type characterization of the trace spaces of %Sobolev spaces to lower dimensional subsets of the Euclidean space.

We study Sobolev type spaces defined in terms of sharp maximal functions on Ahlfors regular subsets of $\R^{n}$ and the relation between these spaces and traces of classical Sobolev spaces.
This extends in a certain way the results of Shvartsman \cite{Shvartsman} to the case of lower dimensional subsets of the Euclidean space.
\end{abstract}
\maketitle

\section{Introduction}

% We consider n kinds of smooth functions defined on Ahlfors-regular subsets of $\R^n$. 

A. Calder\'on proved in \cite{Calderon1972} that a function belongs to the Sobolev space on $\R^{n}$
%$W^{k,p}(\R^n)$ 
if and only if the function and its sharp maximal function of the corresponding order are both in 
an $L^{p}$-space;
%$L^p(\R^n)$.
see also \cite{CalderonScott1978}.  This characterization does not use the notion of derivatives and therefore it can 
be used to at least formally define Sobolev spaces in more general settings.
%This approach has already proved successful.
%This approach has already been successfully applied recently.
Some recent results show that this approach is reasonable. In particular,
P. Shvartsman proved  in \cite{Shvartsman} that the trace of a Sobolev
space to an arbtirary Ahlfors $n$-regular subset of the
Euclidean space admits an intrinsic Calder\'on type characterization. 
Note that this kind of subsets may even have an empty interior. See also
\cite{HajlaszKoskelaTuominen}, where 
a description in terms of sharp maximal functions for Sobolev spaces on extension domains was given.

% the results of Shvartsman were used provide 

%where Sobolev spaces were characterized .... 
%for case of Sobolev extension domains.

%See also \cite{HajlaszKoskelaTuominen},
%where Sobolev spaces were characterized .... 
%for case of Sobolev extension domains.

The main purpose of this paper is to extend Shvartsman's results to lower dimensional closed subsets of the Euclidean space. 
In \cite{Jonsson} A. Jonsson characterized the trace spaces of Sobolev spaces to Ahlfors regular sets as certain Besov type spaces. 
Therefore, our problem can be also formulated as comparison 
of these Besov spaces and Calder\'on type spaces.
Since traces of Sobolev spaces to lower dimensional subsets are of essentially different character than classical Sobolev spaces, 
an exact characterization of Calder\'on type  seems not to be possible on such subsets. 
However, in our main result, Theorem \ref{thm:besov}, we show that Calder\'on type spaces lie between certain Besov spaces. 
This result in particular implies that the trace space of a Sobolev space is embedded in the Calder\'on type space of the corresponding order and that it contains any Calder\'on space of higher order, see Corollary \ref{corollary}.
%For the traces of Sobolev spaces of the first order similar  problems were considered in \cite{HajlaszMartio}. See in particular Theorem 4, where 
Our results have a similar spirit as Theorem 4 in \cite{HajlaszMartio}, where 
relations between the trace spaces of first order Sobolev spaces and Haj\l asz-Sobolev spaces were explored. In particular, their result is also of embedding type, not a sharp characterization, which is not surprising since the Haj\l asz-Sobolev space  coincides with a Calder\'on type space
in their context.

%an interpretation of the classical trace theorem for Sobolev spaces of first order was given in terms of %Haj\l{}asz-Sobolev spaces, since the Haj\l asz-Sobolev space  coincides with a Calder\'on type space
%in their context. 

%. For example, our Theorem \ref{thm:besov} is in a certain way analogue of Theorem 4 in %\cite{HajlaszMartio}, 
%, see for example the proof in \cite{HajlaszKinnunen}. 

%In the first order case, the sharp maximal function of 
%a locally integrable function 
%$f$
%can be defined as
%\begin{equation*}%\label{SharpMaximalFunctions} 
%f^\sharp(x):=\sup\limits_{Q\ni x}\frac{1}{|Q|^{1+1/n}}\int_{Q}|f-f_Q|\;dy,\,\,\,\,\,x\in \R^n,
%\end{equation*}
%where the supremum is taken over all cubes $Q\subset\R^n$ containing point $x$, and $f_Q$ is the average value of function $f$ over the cube $Q$. 
%In the general case, when $k\in \N$, the sharp maximal function can be defined using the projection from $L^1(Q)$ onto the space of 
%polynomials of degree $k-1$.

%The main purpose of this paper is 
%to extend in a certain way the results from \cite{Shvartsman} 
%to describe the relationship between the trace spaces of 
%higher order 
%Sobolev spaces $W^{k,p}(\mathbb{R}^n)$ to Ahlfors-regular subsets of $\R^n$, $n-1<s<n$, and the spaces of functions defined in terms %of sharp maximal 
%functions on $S$. 
%The trace spaces were characterized as certain Besov-type spaces by A. Jonsson in \cite{Jonsson}.
% defined on $S$, 
%Therefore the problem can be also formulated as the
%comparison of these Besov spaces with Calder\'on-type smoothness spaces.

In this paper, we only consider Ahlfors regular sets whose codimension is less than one.
This lower bound for the dimension was due to 
the observation that usual properties of sharp maximal functions on $\mathbb{R}^n$, studied e.g. in \cite{DeVoreSharpley84},
remain valid on sets preserving Markov's
inequalities for polynomials,
%led us to the lower bound for the dimension of subsets under consideration, 
and  by \cite{JonssonWallin1984} Ahlfors $s$-regular sets  with $n-1<s\le n$ have this property. 
This class of sets includes, for example, 
many interesting Cantor type sets and self-similar sets. 
 
There are very few approaches to spaces of higher order smoothness even on such kind of sets,
in spite of the fact that the first order smoothness spaces have been extensively studied in different situations. One of the goals of the paper is to show the advantage of Calder\'on's approach or, more precisely, its local polynomial approximation interpretation in \cite{Brudnyi}, \cite{DeVoreSharpley84} and \cite{Shvartsman}, in defining Sobolev type spaces in more general setting; see the related discussion in Section \ref{sect:Sobolev}.

\section{Preliminaries}\label{sect:prel}

Let $H^s$ denote the $s$-dimensional Hausdorff measure on $\mathbb{R}^n$ and let
$Q=Q(x,r)$ be a closed cube in $\mathbb{R}^n$ centered at $x$ with
side length $2r$ and sides parallel to the coordinate axes.

We say that a subset $S\subset\mathbb{R}^n$ is an $s$-set (or Ahlfors $s$-regular)
%(or Ahlfors $s$-regular) 
if there are
%A subset $S\subset\mathbb{R}^n$ is called an \emph{$s$-set} if there are
constants $c_1,\,c_2>0$ such that for every cube $Q=Q(x,r)$ with
center at $S$ and ${\rm diam}\;Q\le {\rm diam}\;S$, we have
\[
c_1r^s\le H^s(Q(x,r)\cap S)\le c_2r^s.
\]
In this paper, we will always assume that $S\subset \R^n$ is an $s$-set with $n-1< s\leq n$.

\subsection{Sobolev spaces}

Let $L^p(\R^n)$ be the Lebesgue space of $L^p$-integrable functions in $\R^n$. 
For a non-negative integer $k$ and $1\leq p\leq \infty$, 
the Sobolev space $W^{k,p}(\R^n)$ consists of all functions $f\in L^p(\R^n)$ 
having distributional derivatives $D^jf$, $|j|\leq k$, in $L^p(\R^n)$.

There are several approaches to the notion of smoothness spaces of fractional order. One of them is as follows.
\subsection{Potential spaces}
The Bessel kernel of order $\alpha>0$ is the function $G_\alpha\in L^1(\R^n)$ defined by
\[
\hat G (\xi)=(1+4\pi^2|\xi|^2)^{-\alpha/2}.
\]
The potential space $L^p_\alpha(\R^n)$, $\alpha\geq 0$, $1\leq p\leq \infty$, is 
\[
L^p_\alpha(\R^n)=\{f=G_\alpha \ast g\,:\, g\in L^p(\R^n)\}, \quad \alpha>0,
\]
and $L^p_0(\R^n)=L^p(\R^n)$. 

It was shown already by Calder\'on \cite{Calderon1961} that if $1<p<\infty$ and $\alpha$ is a non-negative integer, the Sobolev spaces and potential spaces coincide, i.e.
\begin{equation*}
L^p_k(\R^n)=W^{k,p}(\R^n), \,\,\,\, 1<p<\infty,\,\,\,k\in\N.
\end{equation*}

\subsection{Besov spaces}
Another scale of spaces which is widely used in the study of fractional order smoothness properties of functions is the family of Besov spaces.
For $\alpha>0$ and $1\le p,q\le\infty$, the Besov space $B^{p,q}_{\alpha}(\R^n)$ may be defined in the following way. 
Let $k$ be the integer such that $0\le k<\alpha\le k+1$. Then $B^{p,q}_{\alpha}(\R^n)$ consists of functions $f\in L^p(\R^n)$ such that
\begin{equation}\label{DefofBesovSpace}
\sum\limits_{|j|\le k}\Vert D^jf\Vert_p+\sum\limits_{|j|=k}\bigg(\int_{\R^n}\frac{\Vert D^jf(\cdot +h)-D^jf(\cdot)\Vert_p^q}{|h|^{n+(\alpha-k)q}}\,dh\bigg)^{1/q}<\infty, 
\end{equation}
if $k<\alpha<k+1$ and $1\le p,q<\infty$. If $q=\infty$, then  \eqref{DefofBesovSpace} shall be interpreted in the usual limiting way 
and if $\alpha=k+1$, the first difference of $D^jf$ in \eqref{DefofBesovSpace} shall be replaced by the second difference. For more details, see \cite{BesovIlinNikolski}.

There are several equivalent characterizations of Besov spaces $B^{p,q}_{\alpha}(\R^n)$, for a general theory of these spaces see, 
for example, monographs \cite{BesovIlinNikolski}, \cite{Triebel} and the references therein. 
%In the context, 
We are interested in Besov spaces 
as spaces of traces of functions from Sobolev or, more general, potential spaces to subsets of $\R^n$. See the next paragraph for more 
precise formulations.

Jonsson and Wallin  \cite{JonssonWallin1978} extended the notion of a Besov space to more general setting. 
They introduced a definition of Besov spaces on general $s$-sets, $0<s\le n$, in $\R^n$.
The definition is rather technical, but when $n-1<s\leq n$,
%and with a single function, 
it admits a more simple formulation, which is based on the local polynomial approximation approach, see Theorem 5 on p. 135 of \cite{JonssonWallin1984}. 
In this paper, we will use this formulation.

%The definition of Besov spaces on a general $s$-set 
%$S\subset\mathbb{R}^n$ which was introduced by A. Jonsson and H. 
%Wallin \cite{JonssonWallin1978} is rather technical one but in case of 
%$n-1<s\le n$, % an element in $B^{p,q}_\beta(S)$ can be identified 
%with a single function $f$ and it admits a more simple formulation 
%(Theorem 5 in p. 135 of \cite{JonssonWallin1984}): % Theorem [5, JW, p.135]}

\begin{definition} Let $S$ be an $s$-set with $n-1<s\le n$. Let $1\le 
p,q\le\infty$ and $\alpha>0$. Then a function $f$ is in the Besov space 
$B^{p,q}_{\alpha}(S)$ 
if $f\in L^p(S)$ and there is a sequence $\{c_\nu\}_{\nu=0}^\infty$, $\sum_\nu 
c_\nu^q<\infty$, such that for every net $\pi$ with mesh size $2^{-\nu}$, 
$\nu=0,1,\dots$, there is a function $P_\pi f\in P_{[\alpha]}(\pi)$ 
satisfying \begin{equation}\label{BesovSpacesViaLocalApproximations} 
\bigg(\int_{S}|f-P_\pi f|^p\;dH^s\bigg)^{1/p}\le 
2^{-\nu\alpha}c_\nu. \end{equation}
\end{definition}

Here $P_k(\pi)$ denotes the set of all functions $g$ such that 
the trace of $g$ to a cube $Q\in\pi$ is a polynomial of degree at 
most $k$, and $[\alpha]$ is the largest integer that is not greater than $\alpha$.

Note that in \cite{JonssonWallin1984}, the definition is stated for $s$-sets preserving  
Markov's inequality.
However, by Theorem 3 on p. 39 of \cite{JonssonWallin1984}, 
all $s$-sets with $s>n-1$ satisfy this condition.

\subsection{Traces of Sobolev functions}

%Let's say a couple of words about the notion of a trace (restriction) 
%to $S\subset\mathbb{R}^n$ of a function $f\in L^1_{\rm loc}(\mathbb{R}^n)$. 
We say that $f$ can be pointwisely defined at $x$ if the limit
\[
\bar{f}(x)=\lim\limits_{r\to 0}  \kint_{Q(x,r)}f(y)\,dy=\lim\limits_{r\to 0} \frac{1}{|Q(x,r)|}\int_{Q(x,r)} f(y)\,dy
\]
exists. By Lebesgue's theorem $f=\bar{f}$ a.e. in $\mathbb{R}^n$.

At every $x\in S$ where $\bar{f}(x)$ exists, we define the trace of a function $f$ to $S$ by 
\[
f|_S(x):=\bar{f}(x).
\]

% If $f\in L^p_\beta(\mathbb{R}^n)$, $1<p<\infty$,  
% then $\bar{f}(x)$ exists at all points of $\mathbb{R}^n$ 
% except for a subset  whose Hausdorff dimension 
% is at most $n-\beta p$, see for example \cite{AdamsHedberg}. Thus the trace of a function 
% $f\in L^p_\beta(\mathbb{R}^n)$ to an $s$-set, $s>n-\beta p$, is well defined
% i.e. $f|_S$ is defined at $H^s$-a.e. point of $S$.

If $f\in L^p_\beta(\mathbb{R}^n)$, $1<p<\infty$, then the Hausdorff dimension of the set of points $x\in\R^n$ where $\bar{f}(x)$ does not exists is at most $n-\beta p$, see for example \cite{AdamsHedberg}. Thus the trace of a function 
$f\in L^p_\beta(\mathbb{R}^n)$ to an $s$-set, $s>n-\beta p$, is well defined
i.e. $f|_S$ is defined at $H^s$-a.e. point of $S$.

The next statement, which was proved by A. Jonsson in \cite{Jonsson}, gives a characterization of the trace of the potential space to an $s$-set.

\begin{theorem}\label{thm:TracesOfSobolevSpacesToSsets}
Let $S$ be an $s$-set, $0<s<n$, $1<p<\infty$ and 
$\alpha=\beta-\frac{n-s}{p}>0$. Then 
\[
L^p_\beta(\mathbb{R}^n)|_S=B^{p,p}_\alpha(S),
\]
where the equality 
means that the trace operator $\mathcal{R}:f\mapsto f|_S$ satisfies the 
inequality
\[
\Vert \mathcal{R}f\Vert_{B^{p,p}_\alpha(S)}\le c\Vert 
f\Vert_{L^p_\beta(\mathbb{R}^n)}
\]
for some constant $c$ and for all 
functions $f\in L^p_\beta(\mathbb{R}^n)$, and there exists an 
extension operator $\mathcal{E}:B^{p,p}_\alpha(S)\to 
L^p_\beta(\mathbb{R}^n)$ such that for some constant $c$, we have
\[ 
\Vert\mathcal{E}g\Vert_{L^p_\beta(\mathbb{R}^n)}\le c\Vert g\Vert_{B^{p,p}_\alpha(S)}
\]
for all functions $g\in 
B^{p,p}_\alpha(S)$.
\end{theorem}

Since for $1<p<\infty$ and for nonnegative integers $\beta$, the potential space $L^p_\beta(\mathbb{R}^n)$ coincides with the 
Sobolev space $W^{\beta,p}(\mathbb{R}^n)$, the theorem above, in particular, gives a characterization for traces of Sobolev spaces.

\section{The sharp maximal functions on $s$-sets and corresponding smoothness spaces}\label{SecSharpMaximalFunctions}
\label{sect:poly}
In this section, we introduce Calder\'on type smoothness spaces on $s$-regular subsets of $\R^n$ 
which are defined in terms of fractional sharp maximal functions. We start with the basic notions.

\subsection{Local polynomial approximations}
Let $f\in L^u(S,H^s)$, $0<u\le\infty$, $Q$ be a cube in $\R^n$ and $Q_S=Q\cap S$.
Then \emph{the normalized local best 
approximation} of $f$ on $Q$ in $L^u(S)$ norm is
\[
 \mathcal{E}_k(f,Q)_{L^u(S)}:=\inf\limits_{p\in 
P_{k-1}}\bigg(\kint_{Q_
S}|f-p|^u\;dH^s\bigg)^{1/u},
\]
where $P_{k}$, $k\geq 0$, is a family 
of all polynomials on $\mathbb{R}^n$ of degree at most $k$. We also 
set $P_{-1}:=\{0\}$.

%Note a simple property of $\mathcal{E}_k(f,Q)_{L^u(S)}$ as a cube
%function: for every two cubes $Q_1\subset Q_2$
Note that for every pair of cubes such that $Q_1\subset Q_2$, we have
\begin{equation}\label{eqMonotonyOfLocalApproximation}
\mathcal{E}_k(f,Q_1)_{L^u(S)}\le \bigg(\frac{H^s(Q_2\cap
S)}{H^s(Q_1\cap
S)}\bigg)^{1/u}\mathcal{E}_k(f,Q_2)_{L^u(S)}
\end{equation}
and by the $s$-regularity of a set $S$, this further implies that
\begin{equation}\label{eqMonotonyOfLocalApproxSset}
\mathcal{E}_k(f,Q_1)_{L^u(S)}
\le c\bigg(\frac{r_2}{r_1}\bigg)^{s/u}
%\bigg(\frac{H^s(Q_2\cap
%S)}{H^s(Q_1\cap
%S)}\bigg)^\frac{1}{u}
\mathcal{E}_k(f,Q_2)_{L^u(S)}.
\end{equation}
Here $Q_i=Q(x_i,r_i)$.

In the setting of the Euclidean space, $\mathcal{E}_k(f,Q)_{L^u(\R^n)}$ is the main object of the theory of local polynomial approximation and, 
in particular, it gives a unified framework for the description of various spaces of smooth functions, 
see for example the survey \cite{Brudnyi}.

\subsection{Maximal functions}
Fix $\alpha>0$ and set $k=-[-\alpha]$, i.e. the greatest integer strictly less than $\alpha+1$. For a locally integrable 
function $f$ on $S$, we define \emph{the fractional sharp maximal function }
\begin{equation}\label{SharpMaximalFunctions} 
f^\sharp_{\alpha,u,S}(x):=\sup\limits_{t>0}\frac{1}{t^\alpha}\mathcal{E}_k(f,Q(x,t))_{L^u(S)},\,\,x\in 
S.
\end{equation}
From now on, we will write $f^\sharp_{\alpha,S}$ instead of $f^\sharp_{\alpha,1,S}$ for short.

Since $S$ is an $s$-set, it follows
from \eqref{eqMonotonyOfLocalApproxSset} that the supremum over cubes centered at $x$ in the definition above can be 
replaced by the supremum over all cubes with centers in $S$ containing 
point $x$.
%
%we take $k=\max\{l\in \mathbb{Z}: l<\alpha+1\}$, we will get another 
%variant of fractional maximal function. We will denote this maximal function by $f^\flat_{\alpha,S}$. 
%Clearly, $f^\flat_{\alpha,S}$ differs 
%from $f^\sharp_{\alpha,S}$ only if $\alpha$ is a integer.
%Define also corresponding functional spaces
% \begin{remark}
%\end{remark}

When $S=\R^n$, maximal functions of this type were first introduced by Calder\'{o}n \cite{Calderon1972} (see also the paper of Calder\'{o}n and Scott \cite{CalderonScott1978}). It follows from the results of \cite{Calderon1972} that a function belongs to the Sobolev space $W^{k,p}(\R^n)$, $1<p<\infty$, if and only if $f$ and $f^\sharp_{k,\R^n}$ are both in $L^p(\R^n)$.

Motivated by Calder\'{o}n's characterization of Sobolev spaces define the following function spaces on $s$-sets
\begin{equation}\label{eqGeneralizedSobolevClasses}
C_{\alpha}^p (S)=\{f \in L^p(S):\Vert f  \Vert_{C_{\alpha}^{p}} =
\Vert f \Vert_{p} + \Vert f^\sharp_{\alpha,S}
\Vert_{p}<\infty\},\ \ p\geq 1.
\end{equation}

\begin{remark}\label{RmFlatMaximalFunction} 
If, in the definition \eqref{SharpMaximalFunctions}, we make another choice for the degree of projection, namely, 
 we set $k=[\alpha]+1$ i.e. the smallest integer that is strictly larger than $\alpha$, 
we will get another variant of a fractional maximal function. We will denote it by $f^\flat_{\alpha,S}$. 
Clearly, $f^\flat_{\alpha,S}$ differs from $f^\sharp_{\alpha,S}$ only if $\alpha$ is an integer.
\end{remark}
%\begin{equation*}
%\mathcal{C}_{\alpha}^p (S)=\{f \in L^p(S):\Vert f  
%\Vert_{\mathcal{C}_{\alpha}^{p}} =
%\Vert f \Vert_{p} + \Vert f^\flat_{\alpha,S}
%\Vert_{p}<\infty\},\ \ p\geq 1.
%\end{equation*}

Fractional sharp maximal functions on $\R^n$ and the corresponding smoothness spaces were studied in detail in the monograph of 
R. DeVore and R. Sharpley \cite{DeVoreSharpley84}. Note that in this paper, we use the same notation as \cite{Shvartsman},  
but it differs from the one in \cite{DeVoreSharpley84}.

P. Shvartsman proved in \cite{Shvartsman} that, when $S$ is an $n$-regular subset of $\R^n$, the trace space to $S$ of the Sobolev space
can be characterized via sharp maximal functions, namely,
\[
W^{k,p}(\mathbb{R}^n)|_S=C_{k}^p (S),\ \ p>1.
\]

We aim to study the relationship between the trace spaces of
$W^{k,p}(\mathbb{R}^n)$ to an $s$-set $S$, $n-1<s< n$, and the
spaces of functions  defined in terms of sharp maximal functions
on $S$. Since in this case the trace space
$W^{k,p}(\mathbb{R}^n)|_S$ coincides with the Besov space
$B^{p,p}_{\alpha}(S)$, $\alpha=k-\tfrac{n-s}{p}>0$ (see Theorem \ref{thm:TracesOfSobolevSpacesToSsets}), 
the problem can be also formulated as the
comparison of $B^{p,p}_{\alpha}(S)$ with
%$\mathcal{C}_{\alpha}^p(S)$ 
$C_{\alpha}^p(S)$.
Note that in this case, $\alpha$ is not an integer and consequently the exact choice of $k$ 
%in case $\alpha$ is an integer 
for integer $\alpha$ does not matter, 
see Remark \ref{RmFlatMaximalFunction}.

%\subsection{Linearization of the local approximations}
\subsection{Projectors}
For the study of sharp maximal functions \eqref{SharpMaximalFunctions}, it is useful to construct for every cube $Q\subset \R^n$ 
a projection operator $P_Q$ from $L^1(Q\cap S)$ onto the subspace $P_{k-1}(\R^n)|_{Q\cap S}$, $k\in \N$, such that
\begin{equation*}%\label{eqBestPolynomial}
\mathcal{E}_k(f,Q)_{L^u(S)}\approx (H^{s}(Q\cap S))^{-1/u}\Vert f-P_Qf\Vert_{L^u(Q\cap S)}.
\end{equation*}

%Using the following inequalities for polynomials \cite{JonssonWallin??} 
%(see also \cite{BrudnyiBrudnyi}) one could show that sharp maximal functions 
%$f^\sharp_{\alpha,S}$, 
%where $S$ is an $s$-set, $n-1< s\le n$, have
%the most of the properties of sharp maximal functions 
%$f^\sharp_{\alpha,\mathbb{R}^n}$

This is possible due to the following property of polynomials. 
%on subsets of $\R^n$.

%For the proof of the following statement, see Proposition 3 on p.36 in \cite{JonssonWallin1984}. See also \cite{BrudnyiBrudnyi} where a more general 
%inequality is proved.

\begin{proposition}\label{prop:RevHolderForS-set}
Let $S$ be an $s$-set with
$n-1<s\le n$ and $1\le q,u\le\infty$. Then for every polynomial $p$ 
of degree $k$ and every cube $Q$ centered at $S$, we have
\begin{equation}\label{ReverseHolderForPolynomials}
\bigg(\kint_{Q\cap S}
|p|^q\;dH^s\bigg)^{1/q}\le c\bigg(
\kint_{Q\cap S}|p|^u\;dH^s\bigg)^{1/u},
\end{equation}
where the constant $c>1$ depends on $n$, $k$ and $S$.
\end{proposition}
See Proposition 3 on p. 36 in \cite{JonssonWallin1984} for the proof. See also \cite{BrudnyiBrudnyi}, where a more general inequality of such kind is proved.

Actually, the reverse H\"older inequality for polynomials \eqref{ReverseHolderForPolynomials} guarantees that the
maximal functions $f^\sharp_{\alpha,S}$ have most of the properties of their counterparts defined on $\R^n$. In particular, 
we use it to show that, in the definition of the space $C^p_\alpha(S)$, the function $f^\sharp_{\alpha,S}$ can be replaced with
$f^\sharp_{\alpha,u,S}$, $1<u\le p$, without changing the space.

%Reverse H\"older inequality for the polynomials
%can be used to show that 
%maximal functions $f^\sharp_{\alpha,S}$ have 
%most of the properties of the corresponding maximal functions 
%on $\R^n$, if $S$ is Ahlfors $s$-regular with $n-1<s\le n$. In particular, we use these results to show that in the definition of the space %$C^p_\alpha(S)$,
%the maximal function $f^\sharp_{\alpha,S}$ can be replaced with
%some other variants of sharp maximal functions without changing the space.

Recall that $Q_S=Q\cap S$. We fix now one more notation, namely, for a cube $Q$ and a function $f\in 
L^u(S,H^s)$, $1\le u\le\infty$, we denote
\[
E_k(f,Q)_{L^u(S)}:=\inf\limits_{p\in 
P_{k-1}}\bigg(\int_{Q_S}|f-p|^u\;dH^s\bigg)^{1/u}.
\]
%thus,
%\[\mathcal{E}_k(f,Q(x,r))_{L^u(S)}\asymp
%\bigg(\frac{1}{H^s(Q\cap S)}\bigg)^{\frac{1}{u}}
%c_1^{-1/u}
%r^{-s/u}E_k(f,Q(x,r))_{L^u(S)}.\]
%where $Q=Q(x,r)$.

%Now we are ready to show the existence of a projection operator.

\begin{proposition} \label{prop:Projection}
Let $k\in \N$ and $Q$ be a cube centered at $S$. Then there exists a linear operator $P_Q:L^1(Q_S)\to 
P_{k-1}$ such that for every $1\le u\le \infty$ and every $f\in 
L^u(S)$ 
\[
\bigg(\int_{Q_S}|f-P_Qf|^u\;dH^s\bigg)^{1/u}\le 
cE_k(f,Q)_{L^u(S)},
\]
with some constant $c$ independent of $Q$.
\end{proposition}

\begin{proof}
Following the construction of $P_Q$ from \cite{Shvartsman}, let 
$\{p_\beta:|\beta|\le k-1\}$ denote an orthonormal basis in the linear 
space $P_{k-1}$ with respect to the inner product 
\begin{equation}\label{InnerProduct}
\langle 
f,g\rangle=\int_{Q_S}fg\;dH^s.
\end{equation}
Note that since $s>n-1$, formula \eqref{InnerProduct} defines an inner product indeed.
Set 
\[
P_Qf:=\sum_{|\beta|\le 
k-1}\bigg(\int_{Q_S}fp_\beta\;dH^s\bigg)p_\beta.
\]
We estimate the 
operator norm of $P_Q$ in $L^u$ norm. For every $f\in L^u(Q_S)$, we have
\[
\Vert P_Qf\Vert_{L^u(Q_S)}\le\sum_{|\beta|\le 
k-1}\bigg|\int_{Q_S}f p_\beta\;dH^s\bigg|\Vert p_\beta\Vert_{L^u(Q_S)}.
\] 
By the H\"older inequality,
\[
\Vert 
P_Qf\Vert_{L^u(Q_S)}
\le\bigg(\sum_{|\beta|\le k-1}\Vert 
p_\beta\Vert_{L^u(Q_S)}\Vert p_\beta\Vert_{L^{u'}(Q_S)}\bigg)\Vert 
f\Vert_{L^u(Q_S)},
\]
and by Proposition~\ref{prop:RevHolderForS-set}, 
\[
\begin{split}
 \Vert 
& p_\beta\Vert_{L^u(Q_S)}\Vert p_\beta\Vert_{L^{u'}(Q_S)}\\
\le 
c&\big((H^s(Q_S))^{\frac{1}{u}-\frac{1}{2}}\Vert 
p_\beta\Vert_{L^2(Q_S)}\big)\big((H^s(Q_S))^{\frac{1}{u'}-\frac{1}{2}}\Vert 
p_\beta\Vert_{L^2(Q_S)}\big)=c.
\end{split}
\] 
Hence 
\[
\Vert P_Qf\Vert_{L^u(Q_S)}\le c\Vert f\Vert_{L^u(Q_S)}.
\]

Now, let $p_Q$ denote a polynomial of degree $k-1$ satisfying 
\[
\bigg(\int_{Q_S}|f-p_Q|^u\;dH^s\bigg)^{1/u}=E_k(f,Q)_{L^u(S)}.
\]
% Such a polynomial exists, because the space of polynomial of degree $k-1$ is finite dimensional, 
% and therefore it contains at least one point of minimum distance.
Then we can write
\[
f-P_Qf=(f-p_Q)-P_Q(f-p_Q)
\] 
and, consequently, we get the estimate
\[
\begin{split}
\Big(\int_{Q_S}|f-P_Qf|^u\;dH^s\Big)^{1/u}
&\le(1+\Vert P_Qf\Vert_{L^u(Q_S)})E_k(f,Q)_{L^u(S)}\\
&\le c E_k(f,Q)_{L^u(S)}.\qedhere
\end{split}
\]
\end{proof}

% To every cube $Q$ centered at $S$, we assign a 
% linear mapping $P_Q:L^1_{\rm loc}(Q_S)\to P_{k-1}$ by setting 
% \[
% P_Qf:=P_Q(f).\qedhere
% \]

The proposition above together with the definition of the sharp maximal function \eqref{SharpMaximalFunctions} implies that
%Thus, by the definition of the sharp maximal function \eqref{SharpMaximalFunctions} and Proposition \ref{prop:Projection}, we have
\begin{equation}\label{defSharpMaximalViaProjectors}
f^\sharp_{\alpha,u,S}(x)\approx \sup\limits_{t>0}\frac{1}{t^\alpha}\bigg(\kint_{Q(x,t)\cap S}|f-P_{Q(x,t)}f|^u\;dH^s\bigg)^{1/u}.
\end{equation}

Now we consider some properties of the projectors $P_Q$. 

%Then we will show that the set of functions 
%such that $f^\sharp_{\alpha,u,S}\in L^p(S)$
%is independent of $u$ as long as $1\leq u\leq p$, 
%see Remark \ref{remark:OnDefinitionofCpa}. Thus,
%we can use any value of $u$ in the definition 
%of the space $C^p_\alpha(S)$.

% goal: remark 3.10

\begin{lemma}\label{lemma:ProjectProperties} Let function $f\in L^1_{\rm loc}(Q\cap S)$ and cube $Q=Q(x,r)\subset \R^n$ be centered at $x\in S$, then:
\begin{enumerate}
\item\label{item1} $P_Q(\lambda)=\lambda$ for any $\lambda\in \mathbb{R}$;
\item\label{item2} $|P_Q f(y)|\leq c|f|_{Q\cap S}$, \,\,$y\in Q\cap S$;

\item\label{item3} If $Q'$ centered at $S$ is such that
$Q'\subset Q$ and 
\[
H^s(Q'\cap S)\geq cH^s(Q\cap S),
\]
then
\[
| P_Qf(z)-P_{Q'}f(z)|
\le c\kint_{Q\cap S}|f-P_Q|\,dH^s,\,\,\,z\in{Q'\cap S};
\]
\item\label{item4} If $Q'=Q(y,r)$, $y\in S$, such that
$Q'\cap Q\neq\emptyset$, then
\[
|P_Qf(z_1)-P_{Q'}f(z_1)|\le
c
\kint_{Q(z_2,2r)\cap S}|f-P_{Q(z_2,2r)}f|\,dH^s
 \]
for every $z_1,z_2\in Q\cap Q'\cap S$.
\end{enumerate}
\end{lemma}
\begin{proof}
Properties \eqref{item1} and \eqref{item2} directly follow from the construction of projectors $P_Q$. Let us prove \eqref{item3}. 
By \eqref{ReverseHolderForPolynomials}, we have
\begin{equation*}%\label{eqPolynomialDifference}
\begin{split}
\sup\limits_{Q'_S}&| P_Qf-P_{Q'}f|
\le \,c\kint_{Q'_S}|P_Qf-P_{Q'}f|\;dH^s\\
&\le \,c\bigg[\kint_{Q'_S}|f-P_{Q}f|\;dH^s+\kint_{Q'_S}|f-P_{Q'}f|\;dH^s\bigg]
\\
&\le \,c 
\bigg[\kint_{Q_
S}|f-P_{Q}f|\;dH^s+\mathcal{E}_k(f,Q')_{L^1(S)}\bigg]\le c\kint_{Q_S}|f-P_{Q}f|\;dH^s.
\end{split}
\end{equation*}

Note that if cubes $Q=Q(x,r)$ and $Q'=(y,r')$ are such that $Q\cap Q'\neq\emptyset$  then $Q,Q'\subset Q(z_2,2r)$ for any $z_2\in Q\cap Q'$. Then, since
\[
\begin{split}
&|P_Qf(z_1)-P_{Q'}f(z_1)|\\
\le& |P_Qf(z_1)-P_{Q(z_2,2r)}f(z_1)|+|P_{Q'}f(z_1)-P_{Q(z_2,2r)}f(z_1)|,
\end{split}
\]
the statement \eqref{item4} easily follows from \eqref{item3}.
\end{proof}

\begin{remark}
If $x\in S$ is a Lebesgue point of a function
$f\in L^1_{\rm loc}(S)$, then, by definition,
\begin{equation}\label{eqLebesguePointThroughtDifference}
\lim\limits_{r\to 0}\kint_{Q(x,r)\cap S}|f-f(x)|\,dH^s=0.
\end{equation}
By statements \eqref{item1} and \eqref{item2} of Lemma \ref{lemma:ProjectProperties}, we have
\begin{equation*}
|P_{Q(x,r)}f(x)-f(x)|=|P_{Q(x,r)}[f-f(x)](x)|\leq
c\kint_{Q(x,r)\cap S}|f-f(x)|dH^s.
\end{equation*}

Since almost every point of $S$ is a Lebesgue point of a function $f\in L^1_{\rm loc}(S)$ (see for example \cite{Heinonen}), we have
\begin{equation}\label{CovergenceOfP_QatPoint}
\lim\limits_{r\to 0}P_{Q(x,r)}f(x)=f(x)\,\,\,\text{a.e. on}\,\,S. 
\end{equation}
\end{remark}

The following lemma is a special case of Theorem 1 in \cite{Ihnatsyeva} (see also Theorem 1 in \cite{IvanishkoKrotov}). 
For the sake of completeness, we will sketch the proof here.

\begin{lemma}\label{lemmaSobolevPoincareIneguality}
Suppose that $\alpha>0$, $q\geq 1$ and $f\in
L_{\rm loc}^{1}(S)$. Then for any cube $Q=Q(x,r)$, $x\in S$, we have
\begin{equation}
\bigg(\kint_{Q\cap S}|f-P_Qf|^q\,dH^s\bigg)^{1/q}\leq c
r^\alpha \bigg(\kint_{2Q\cap S}(f^\sharp_{\alpha,S})^\sigma\,dH^s\bigg)^{1/\sigma},
\end{equation}
where $\tfrac{1}{\sigma}=\tfrac{1}{q}+\tfrac{\alpha}{s}$.
\end{lemma}

\begin{proof}
Let $x_0$ be a Lebesgue point of a function $f$. We will show that 
\begin{equation}\label{eqRateOfSteklovMeansApproximationWithS}
|f(x_0) - P_{Q(x_0,r)}f(x_0)| \le c r^\alpha\bigg(
f^\sharp_{\alpha,S}(x_0)\bigg)^{1-\alpha \sigma /s}
\bigg(\kint_{Q_S(x_0,r)}(f^\sharp_{\alpha,S} )^\sigma \, dH^s \bigg)^{\alpha/s}.
\end{equation}
For every cube $Q(x,t)\subset \R^n$, let $Q_S(x,t)$ denote the set $Q(x,t)\cap S$ and consider
\begin{equation*}%\label{}
u(x,t)=\frac{1}{t^{\alpha}}
\kint_{Q_S(x,t)}|f - P_{Q(x,t)}f|
\, dH^s,
\end{equation*}
By Proposition \ref{prop:Projection} and \eqref{eqMonotonyOfLocalApproxSset}, we have
\begin{equation}\label{eqGrowthOfU}
u(x,\tau) \le c u(x,t) \;\;\text{if} \; \tau \le t \le 2 \tau.
\end{equation}
If $Q_k = Q(x_0, 2^{-k} r)$, $k\ge0$, then by \eqref{CovergenceOfP_QatPoint}, Lemma \ref{lemma:ProjectProperties} and \eqref{eqGrowthOfU} respectively,
we obtain
\begin{equation}\label{eqIntegralRepresentation}
\begin{split}
|f(x_0) - P_{Q(x_0,r)}f(x_0)| 
&= \left|\sum_{k=0}^\infty (P_{Q_{k+1}}f(x_0) - P_{Q_k}f(x_0)) \right|
\\
&\le c r^\alpha \sum_{k=0}^\infty 2^{-k\alpha} u (x_0, 2^{-k} r)\\
&\le
%c r^\alpha \big(u(x_0,r)+\sum
%\limits_{k=1}^\infty 2^{-k\alpha} u (x_0, 2^{-k} r)\big)\le
%\end{equation*}
%\begin{equation*}
%\le cr^\alpha u(x_0,r)+c\sum \limits_{k=1}^\infty
%\int_{2^{-k}r}^{2^{-(k-1)}r} t^{\alpha} u(x_0,t) \frac{d
%t}{t}\le
%\end{equation*}
c r^\alpha u(x_0,r)+c\int_{0}^{r}t^{\alpha} u(x_0,t)
\frac{d t}{t}.
\end{split}
\end{equation}
Let
\begin{equation}\label{eqMeanOfSharpMaxFunction}
I=\bigg(\kint_{Q_S(x_0,r)}
(f^\sharp_{\alpha,S})^\sigma \,dH^s \bigg)^{1/\sigma}
\end{equation}
and consider two cases:
\begin{enumerate}
\item \label{case1} If $f^\sharp_{\alpha,S}(x_0)\le I $ then by \eqref{defSharpMaximalViaProjectors}, we have
\begin{equation*}%\label{eqI<NintEstimate}
\int_{0}^{r} t^{\alpha-1} u(x_0,t)\,d t \le c
r^\alpha f^\sharp_{\alpha,S}(x_0)
\le c r^\alpha (f^\sharp_{\alpha,S}(x_0))^{1-\frac{\alpha \sigma}{s}} I^{\frac{\alpha
\sigma}{s}}.
\end{equation*}
\item If $f^\sharp_{\alpha,S}(x_0)>I$, then define
$\tau=rI^{\sigma/s}(f^\sharp_{\alpha,S})^{-\sigma/s} < r$ and write
\begin{equation*}%\label{}
\int_{0}^{r} t^{\alpha-1}u(x_0,t)\,d t= \bigg(\int_{0}^{\tau} +
\int_{\tau}^{r}\bigg)t^{\alpha-1}u(x_0,t)\,dt\equiv I_1+I_2.
\end{equation*}
Then
\begin{equation*}%\label{eqSecondCaseEstimateForI1}
I_1\le c\tau^\alpha f^\sharp_{\alpha,S}(x_0)=c r^\alpha (f^\sharp_{\alpha,S}(x_0))^{1-\frac{\alpha \sigma}{s}} I^{\frac{\alpha
\sigma}{s}}.
\end{equation*}
To estimate $I_2$, note that for every $t\leq r$, we have
\begin{equation}\label{eqAveragePrinciple}
u(x_0,t) \le c\bigg( \kint_{Q_S(x_0,t)} (f^\sharp_{\alpha,S}) ^\sigma \, dH^s \bigg)^{1/\sigma}\le c
\bigg(\frac{r}{t}\bigg)^{s/\sigma}I,
\end{equation}
and therefore
\begin{equation*}%\label{}
I_2\le c Ir^{s/\sigma}\int_{\tau}^{r} t^{\alpha-1-s/\sigma}\, dt
\le c I\bigg(\frac{r}{\tau}\bigg)^{s/\sigma}\tau^\alpha\le cr^\alpha I,
\end{equation*}
Consequently, we have the same estimate as in case \eqref{case1} for the integral in \eqref{eqIntegralRepresentation}. 
\end{enumerate}

To estimate $u(x_0,r)$, we use \eqref{eqAveragePrinciple} with $t=r$. Thus
\begin{equation*}%\label{}
u(x_0,r)=[u(x_0,r)]^{1-\frac{\alpha
\sigma}{s}}[u(x_0,r)]^{\frac{\alpha \sigma}{s}}\le
[f^\sharp_{\alpha,S}(x_0)]^{1-\frac{\alpha
\sigma}{s}}I^{\frac{\alpha \sigma}{s}},
\end{equation*}
which finishes the proof of \eqref{eqRateOfSteklovMeansApproximationWithS}. 

Now consider
\begin{equation*}%\label{}
\begin{split}
\bigg(\kint_{Q_S}|f - P_Q f|^q \,dH^s\bigg)^{1/q}
\le
&\bigg(\kint_{Q_S}|f(y) -
P_{Q(y,r)}f(y) |^q \, dH^s(y)\bigg)^{1/q} \\
&+\bigg(\kint_{Q_S}|P_{Q(y,r)}f(y)-P_Q
f(y)|^q \,dH^s(y)\bigg)^{1/q}\\
\equiv& I_1+I_2.
\end{split}
\end{equation*}
By \eqref{eqRateOfSteklovMeansApproximationWithS}, we have
\begin{equation*}%\label{}
I_1\le c r^\alpha\bigg(\kint_{Q_S}(f^\sharp_{\alpha,S}(y))^{q(1-\alpha \sigma/s)}
\Big(\kint_{Q_S(y,r)}(f^\sharp_{\alpha,S} )^\sigma\,dH^s
\Big)^{q\alpha/s}dH^s(y) \bigg)^{1/q},
\end{equation*}
and since for every $y\in Q(x,r)$, the cube $Q(y,r)\subset Q(x,2r)$, we have
\begin{equation*}
\begin{split}
I_1&\le c r^\alpha
\bigg(\kint_{Q_S(x,2r)}(f^\sharp_{\alpha,S})^\sigma\,dH^s \bigg)^{\alpha/s}\bigg(
\kint_{Q_S}(f^\sharp_{\alpha,S})^\sigma\,dH^s(y)\bigg)^{1/q} 
\\
&\le
cr^\alpha\kint_{Q_S(x,2r)}(f^\sharp_{\alpha,S})^\sigma\,dH^s\bigg)^{1/\sigma}.
\end{split}
\end{equation*}
If $y\in Q(x,r)\cap S$ and $z\in Q(x,2r)\cap S$, then by statement \eqref{item4} of Lemma \ref{lemma:ProjectProperties}, we have
\begin{equation*}%\label{}
|P_{Q(y,r)}f(y)-P_Q f(y)|\le
c\kint_{Q_S(x,2r)}|f-P_{Q(x,2r)}f|\,dH^s\le c r^\alpha f^\sharp_{\alpha,S}(z).
\end{equation*}
Since the last inequality holds for any $z\in Q_S(x,2r)$, we have
\begin{equation}\label{eqEstimateOfDifferenceOfTwoMeansWithAverage}
|P_{Q(y,r)}f(y)-P_Q f(y)|\le cr^\alpha
\bigg(\kint_{Q_S(x,2r)}
(f^\sharp_{\alpha,S})^\sigma \,dH^s \bigg)^{1/\sigma}
\end{equation}
for every $y\in Q_S$. This completes the proof.
\end{proof}

Applying the H\"{o}lder inequality and Lemma \ref{lemmaSobolevPoincareIneguality} respectively, we get the following statement.
\begin{lemma}\label{lemma:sharpMaximal}
Let $\alpha>0$, $u>1$ and $f\in L^1_{\rm loc}(S)$, then
\begin{equation}\label{EquivalenaceOfFsharpWithExponent}
f^\sharp_{\alpha,S}\le f^\sharp_{\alpha,u,S}(x) \le c M_\sigma (f^\sharp_{\alpha,S})(x),
\end{equation}
where $1/\sigma=1/u+\alpha/s$, $M$ is the Hardy-Littlewood maximal operator and $M_\sigma(g)=[M(|g|^\sigma)]^{1/\sigma }$.
\end{lemma}

%Lemma \ref{lemma:sharpMaximal} is derived 
%from Lemma \ref{lemma:rearranged} using Hardy inequality and properties of the Lorentz 
%norms of a function (see, e.g. Theorem 4.3 in \cite{DeVoreSharpley84}).
\begin{remark}\label{remark:OnDefinitionofCpa}
Recall that the function space $C_{\alpha}^p (S)$ is defined as
the set of functions $f\in L^p(S)$ such that
$f^\sharp_{\alpha,1,S}\in L^p(S)$. By Lemma \ref{lemma:sharpMaximal} and the $L^q$-boundedness of the 
maximal operator for $q>1$, the set of functions such that $f^\sharp_{\alpha,u,S}\in L^p(S)$ is independent of $u$ as long as $1\leq u\leq p$. Thus,
we can use any value of $u$ in the definition of the space $C^p_\alpha(S)$. 
Furthermore, the next lemma shows that to define $C^p_\alpha(S)$, it is enough to consider local best approximations 
on cubes with side length less than any fixed positive number.
%Thus, we can will get the same set of functions if we require that
%$f^\sharp_{\alpha,u,S}\in L^p(S)$ for some $1\le u\le p$.
\end{remark}

\begin{lemma}\label{CharacterizationViaResrictedFunctions}
Let $p>1$, $1\leq u\leq p$ and $\gamma>0$. Then $C^p_\alpha(S)$ coincides with the space 
\[
\{f\in 
L^p(S):\,\sup\limits_{0<t<\gamma}t^{-\alpha}\mathcal{E}_k(f,Q(\cdot,t))_{L^u(S)}\in 
L^p(S)\}.
\]
%where $\gamma>0$ and $1\le u\le p$.
\end{lemma}
\begin{proof}
First let $1\le u<p$. For every $x\in S$, we have
\begin{equation}\label{equivToRestrictedSMfunct} 
\begin{split}
\sup\limits_{t\geq \gamma}\frac{1}{t^\alpha}\mathcal{E}_k(f,Q(x,t))_{L^u(S)}
& \le \gamma^{-\alpha}\sup\limits_{t\geq 
\gamma}\bigg(\kint_{Q(x,t)\cap S}|f|^u\;dH^s\bigg)^{1/u}\\
&\le c (M(f^u)(x))^{1/u}
\end{split}
\end{equation}
and the claim follows from the $L^q$-boundedness of the 
maximal operator for $q=p/u>1$.

If $u=p$, we take some $1\le q<p$. By 
\eqref{EquivalenaceOfFsharpWithExponent}, we have 
\[\begin{split}
&\Vert 
f^\sharp_{\alpha,p,S}\Vert_{L^p(S)}\le c\Vert 
f^\sharp_{\alpha,q,S}\Vert_{L^p(S)}
\\
&\le \Vert\sup\limits_{0<t<\gamma 
}\frac{1}{t^\alpha}\mathcal{E}_k(f,Q(x,t))_{L^q(S)}\Vert_{L^p(S)}+\Vert\sup\limits_{t\geq 
\gamma}\frac{1}{t^\alpha}\mathcal{E}_k(f,Q(x,t))_{L^q(S)}\Vert_{L^p(S)}, 
\end{split}
\]
where the first summand is bounded by the assumption and the second 
by \eqref{equivToRestrictedSMfunct}.
\end{proof}

\section{Comparison with Besov spaces}
\label{sect:main}
%In $n$-regular subsets of $\R^n$, we have $C^p_\alpha(S)=B^{p,p}_\alpha(S)=W^{k,p}(S)$, see \cite{Shvartsman}.

The following theorem is the main result of the present paper.

\begin{theorem}\label{thm:besov}
Let $S$ be an $s$-set with $n-1<s\le n$, 
$1<p\le\infty$ and $\alpha$ be a non-integer positive number. Then \begin{equation}\label{Embedding}
B_\alpha^{p,p}(S)\subset C_\alpha^p(S)\subset 
B_{\alpha}^{p,\infty}(S). \end{equation}
\end{theorem}

\begin{remark}\label{RmEmbeddingsForIntegerAlpha}
If $\alpha>0$ is an integer then in the embeddings \eqref{Embedding} the space  $C_\alpha^p(S)$ shall be replaced with the space 
\[
\{f\in L^p(S): f^\flat_{\alpha,S}\in L^p(S)\},
\]
see Remark \ref{RmFlatMaximalFunction} for the difference between functions $f^\sharp_{\alpha,S}$ and $f^\flat_{\alpha,S}$. 
\end{remark}

Theorem~\ref{thm:besov} is an analogue of Theorem 7.1 in \cite{DeVoreSharpley84} for 
$S=\mathbb{R}^n$. Examples similar to the ones constructed in \cite{DeVoreSharpley84} show that the embeddings \eqref{Embedding} are the best possible within the scale of Besov spaces. 

The case of $s$-set with $s$ strictly less than $n$ 
is of our current interest due to the characterization for traces of potential spaces to $s$-sets given by A. Jonsson, see Theorem \ref{thm:TracesOfSobolevSpacesToSsets}.
%For example, It easily follows from \ref{thm:besov} and basic embeddings for Besov spaces
% Since potential space $L^p_\beta(\mathbb{R}^n)$ coincides with the 
% Sobolev space $W^{\beta,p}(\mathbb{R}^n)$ if $\beta$ is an integer, in 
% particular, Theorem \ref{thm:TracesOfSobolevSpacesToSsets} is a trace theorem for Sobolev spaces.

\begin{corollary}\label{corollary}
Let $S$ be an $s$-set, $n-1<s<n$, $1<p<\infty$, $k\in \N$ and $\alpha=k-(n-s)/p>0$. Then for any $0<\varepsilon<(n-s)/p$ 
\begin{equation} 
C_{\alpha+\varepsilon}^p(S)\subset W^p_k(\mathbb{R}^n)|_S\subset C_\alpha^p(S). \end{equation}
\end{corollary}

To prove Theorem~\ref{thm:besov}, 
we need the following representation for the norm of Besov spaces.

\begin{theorem}\label{thm:besov-norm}
Let $S$ be an $s$-set, $n-1<s\le n$, $\alpha>0$,
$1\le p,q\le\infty$ and $k=[\alpha]+1$. Then, when $q<\infty$, we have
\[
\Vert 
f\Vert_{B^{p,q}_\alpha(S)}\approx\Vert 
f\Vert_{L^p(S)}+\bigg(\int_0^1\bigg(\frac{\Vert\mathcal{E}_k(f,Q(\cdot,t))_{L^p(S)}\Vert_{L^p(S)}}{t^\alpha}\bigg)^q\frac{dt}{t}\bigg)^{1/q}
\]
and
\[
\Vert 
f\Vert_{B^{p,\infty}_\alpha(S)}\approx\Vert f\Vert_{L^p(S)}+\sup\limits_{0<t\le 
1}t^{-\alpha}\Vert\mathcal{E}_k(f,Q(\cdot,t))_{L^p(S)}\Vert_{L^p(S)}.
\]
\end{theorem}

\begin{remark}
Such characterization of Besov spaces is fairly standard, see for example \cite{Brudnyi}, \cite{Triebel} for the case when $S=\R^n$ and \cite{JonssonWallin1984}, \cite{Shvartsman} for the case of $n$-sets.
\end{remark}

\begin{proof}
First, suppose that the right-hand-side is finite. We note that by~\eqref{eqMonotonyOfLocalApproximation}, we 
can replace the integral by the sum 
\begin{equation}\label{eqn:sum}
\sum\limits_{\nu=0}^\infty 2^{\nu 
\alpha 
q}\bigg(\int_S\mathcal{E}^p_k(f,Q(x,2^{-\nu}))_{L^p(S)}\;dH^s(x)\bigg)^{q/p}.
\end{equation}

Take a net $\pi=\{Q_i,i=1,2,\dots\}$ with mesh size $2^{-\nu}$ and let 
$P_\pi f$ be a function from $P_k(\pi)$ which will be chosen later. 
Clearly,
\begin{equation*} 
\int_{S}|f-P_\pi f|^p\;dH^s=\sum\limits_{Q\in\pi}\int_{Q\cap 
S}|f-P_\pi f|^p\;dH^s= \sum\limits_{Q\in\pi'}\int_{Q\cap 
S}|f-P_\pi f|^p\;dH^s, 
\end{equation*} 
where $\pi'=\{Q\in\pi:\;Q\cap S\neq\emptyset\}$.

Set $t= 2^{-\nu-1}$. For any cube $Q=Q(x,t)$ from $\pi'$, choose a point $y\in Q\cap S$ 
and set $K=Q(y,2t)$. Then $Q\subset K$ and
\[
\sum\limits_{Q\in\pi'}\chi_{K}\le c,
\]
where constant $c$ depends only on $n$.

The center of every cube $K$ is in $S$. Hence, by Proposition~\ref{prop:Projection}, there
is a projector $P_{K}:L^1(K\cap S)\to P_{[\alpha]}$ such that
\[
\int_{K\cap S}|f-P_Kf|^p\;dH^s\le c H^s(K\cap S)\mathcal{E}^p_{k}(f,K)_{L^p(S)},
\]
with constant $c$ independent of $f$ and $K$.
Define $P_\pi f(x)=P_{K} f(x)$,
$x\in Q$, and $P_\pi f(x)=0$ if
$x\notin\bigcup\limits_{Q\in\pi'}Q$.
% then $P_\pi f\in
%P_{[\alpha]}$. Note that since $\alpha<k$, we have $[\alpha]\le k-1$.
For any point $z\in K\cap S$ we have $K\subset Q(z,4t)$ and
\begin{equation*}
\begin{split} 
\int_{Q\cap S}|f-P_\pi f|^p\;dH^s&=\int_{Q\cap S}|f-P_Kf|^p\;dH^s
\\
&\le \int_{K\cap S}|f-P_Kf|^p\;dH^s \\
&\le c H^s(K\cap S)\mathcal{E}^p_{k}(f,K)_{L^p(S)}
\\
&\le c H^s(K\cap S)\mathcal{E}^p_{k}(f,Q(z,4t))_{L^p(S)}, 
\end{split}
\end{equation*} 
where the last inequality holds by 
\eqref{eqMonotonyOfLocalApproximation}. Then we integrate the inequality over 
the set $K\cap S$ to obtain
\[
 \int_{Q\cap S}|f-P_\pi f|^p\;dH^s\le c 
\int_{K\cap S}\mathcal{E}^p_{k}(f,Q(z,4t))_{L^p(S)}\;dH^s(z). 
\]
Remember that we set $t=2^{-\nu-1}$. Thus we have
\begin{equation*} 
\begin{split}
\left(\int_{S}|f-P_\pi f|^p\;dH^s\right)^{1/p}
\le&\, c\left(\sum_{Q\in\pi'}\int_{K\cap 
S}\mathcal{E}^p_{k}(f,Q(z,2^{-\nu+1}))_{L^p(S)}\;dH^s(z)\right)^{1/p}\\
\le&\, c\left(\int_{S}\mathcal{E}^p_{k}(f,Q(\cdot,2^{-\nu+1}))_{L^p(S)}\;dH^s\right)^{1/p}. 
\end{split}
\end{equation*}
Let now $c_\nu$ be equal to the last integral multiplied 
by $2^{\nu \alpha }$. Then
\[ 
\sum\limits_{\nu=1}^{\infty}c_\nu^q= 
c\sum\limits_{\nu=1}^{\infty}2^{\nu \alpha 
q}\bigg(\int_S\mathcal{E}^p_k(f,Q(x,2^{-\nu}))_{L^p(S)}\;dH^s(x)\bigg)^{q/p}<\infty, 
\]
so that, by \eqref{eqn:sum}, $f\in B^{p,q}_\alpha(S)$ and the wanted estimate for its 
norm holds.

Suppose now that $f\in B^{p,q}_\alpha(S)$ and $\pi$ is a net with 
mesh size $2t$, $t>0$. Denote by $\pi'$ a family of all cubes $Q$ from 
$\pi$ such that $Q\cap S\neq\emptyset$. If $Q\in\pi'$ and $x\in Q\cap 
S$, then $Q(x,t)\subset 2Q$, $H^s(Q(x,t)\cap S)\approx H^s(2Q\cap S)$ 
and by \eqref{eqMonotonyOfLocalApproximation} 
\[ 
\mathcal{E}^p_k(f,Q(x,t))_{L^u(S)}\le c\,\mathcal{E}^p_k(f,2Q)_{L^u(S)}. 
\] 
Hence, 
\[ 
\begin{split}
\int_S\mathcal{E}^p_k(f,Q(x,t))_{L^u(S)}\;dH^s(x)
=&\sum\limits_{Q\in\pi'} 
\int_{Q\cap S}\mathcal{E}^p_k(f,Q(x,t))_{L^u(S)}\;dH^s(x)\\
\le\,&\, c\sum\limits_{Q\in\pi'}H^s(2Q\cap S)\mathcal{E}^p_k(f,2Q)_{L^u(S)}.
\end{split}
\] 
It is easy to see that family of cubes $\tilde{\pi}=\{2Q:\;Q\in\pi'\}$ 
can be represented as $\tilde{\pi}=\cup_{i=1}^{m}\pi_i$, where $m=2^n$ and every $\pi_i$ is a subfamily of a net 
with mesh size $4t$.

Set $k=[\alpha]+1$ and $t=2^{-\nu}$, $\nu=2,\dots$. Since $f\in 
B^{p,q}_{\alpha}(S)$, there are functions $P_{\pi_i} f\in P_{k-1}$, 
$i=1,\dots,m$, such that
\[
\begin{split}
\int_S\mathcal{E}^p_k(&f,Q(x,2^{-\nu}))_{L^u(S)}\;dH^s(x)
\\
&\le c\sum\limits_{i=1}^{m}\sum\limits_{Q\in\pi_i}H^s(Q\cap 
S)\mathcal{E}^p_k(f,Q)_{L^p(S)}
\\
&\le c\sup\limits_{\pi_i}\sum\limits_{Q\in\pi_i}(H^s(Q\cap 
S))^{1-p/p}\int_{Q\cap 
S}|f-P_{\pi_i} f|^u\;dH^s
\\
&= c \sup\limits_{\pi_i}\sum\limits_{Q\in\pi_i}\int_{Q\cap 
S}|f-P_{\pi_i} f|^p\;dH^s,\\
&\le c\sup\limits_{\pi_i}\int_{ 
S}|f-P_{\pi_i} f|^p\;dH^s
\le c 2^{(-\nu+2)\alpha p}c^p_{\nu-2}
\end{split}
\]
 and, consequently
 \[
\sum\limits_{\nu=2}^\infty2^{\nu \alpha 
q}\bigg(\int_S\mathcal{E}^p_k(f,Q(x,2^{-\nu}))_{L^u(S)}\;dH^s(x)\bigg)^{q/p}\le 
c\sum\limits_{\nu=0}^\infty c_{\nu}^q<\infty.\qedhere
\]
\end{proof}

\begin{proof}[Proof of Theorem~\ref{thm:besov}]
We start with the first embedding and use here the characterization of the
spaces $C^p_\alpha(S)$ given by Lemma
\ref{CharacterizationViaResrictedFunctions}.
By property \eqref{eqMonotonyOfLocalApproximation} of local best 
approximation, we have
\[
\begin{split}
\sup\limits_{0<t\le 
\frac{1}{2}}\frac{1}{t^{\alpha 
p}}\mathcal{E}^p_k(f,Q(x,t))_{L^p(S)}&\le c\sum\limits_{\nu=1}^\infty 
2^{-\nu\alpha p}\mathcal{E}^p_k(f,Q(x,2^{-\nu}))_{L^p(S)}
\\
&\le 
c\int_0^1\frac{\mathcal{E}^p_k(f,Q(x,t))_{L^p(S)}}{t^{\alpha 
p}}\;\frac{dt}{t}.
\end{split}
\]
Thus,
\[
\begin{split}
 \Vert f^\sharp_{\alpha,S}\Vert^p_{L^p(S)}\le &\,
c\int_S\int_0^1\frac{\mathcal{E}^p_k(f,Q(x,t))_{L^p(S)}}{t^{\alpha 
p}}\;\frac{dt}{t}dx
\\
=&\,c\int_0^1\bigg(\frac{\Vert\mathcal{E}_k(f,Q(\cdot,t))_{L^p(S)}\Vert_{L^p(S)}}{t^\alpha}\bigg)^p\frac{dt}{t}. 
\end{split}
\]

For non-integer $\alpha>0$ the number $k=-[-\alpha]$ is strictly greater than $\alpha$, hence, by Theorem \ref{thm:besov-norm} the last term can be estimated by
$\|f\|_{B_\alpha^{p,p}(S)}$.

%To show that the right hand embedding holds write an obvious inequality
To prove the second embedding, we notice that for every $\alpha>0$, $k\in\N$ and $t>0$, we have
\begin{equation}\label{eqRightEmbedding}
\frac{\Vert\mathcal{E}_k(f,Q(\cdot,t))_{L^p(S)}\Vert_{L^p(S)}}{t^\alpha}
\le \Vert\sup\limits_{t>0}
\frac{\mathcal{E}_k(f,Q(\cdot,t))_{L^p(S)}}{t^\alpha}\Vert_{L^p(S)}.
\end{equation}
Setting $k-[-\alpha]$ and taking the supremum over the interval $(0,1]$ in \eqref{eqRightEmbedding} we get 
\[
\|f\|_{B_\alpha^{p,\infty}(S)}\le \Vert f^\sharp_{\alpha,S}\Vert_{L^p(S)}.\qedhere
\]
\end{proof}

Since the statements of Lemmata \ref{lemmaSobolevPoincareIneguality}, \ref{lemma:sharpMaximal}, 
\ref{CharacterizationViaResrictedFunctions} hold true for the sharp maximal functions $f^\flat_{\alpha,S}$ as well, 
the case of integer $\alpha$ can be treated with the slight modification of the last proof; see Remark \ref{RmEmbeddingsForIntegerAlpha}.

\section{Sobolev spaces on $s$-sets}\label{sect:Sobolev}

As mentioned above, the definition \eqref{eqGeneralizedSobolevClasses} of the function space $C^p_k(S)$ yields the Sobolev space $W^{k,p}(\R^n)$  
if $S=\R^n$ \cite{Calderon1972} or the space $W^{k,p}(S)$  if $S$ is a $W^{k,p}$-extension domain \cite{HajlaszKoskelaTuominen}. 
Motivated by these facts, one could ask a natural question: Can the spaces $C^p_k(S)$ be relevant analogs of
classical Sobolev spaces in some more general settings? 

If $S$ is an $n$-set, then $C^p_k(S)$ is the trace space of $W^{k,p}(\R^n)$ to $S$ \cite{Shvartsman} and therefore functions from $C^p_k(S)$ 
possess certain distinctive properties of Sobolev spaces. In the case of $s$-sets with $n-1<s<n$, we can not derive the corresponding properties 
from the trace reasoning. Nevertheless, some results which are known for Sobolev spaces of higher order can be obtained. 
In particular, we have a version of Sobolev-Poncar\'e  inequality given by Lemma \ref{lemmaSobolevPoincareIneguality},
and an analogue of Sobolev embedding theorem which is proved below.
\begin{proposition}
Let $S$ be an $s$-set, $n-1<s<n$,  $p\geq 1$, $k p<s$ and $q=sp/(s-kp)$. Then 
\begin{equation}\label{SobolevEmbedding}
\Vert f\Vert_{L^q(S)}\le c(\Vert f^\sharp_{\alpha,S}\Vert_{L^p(S)}+(\diam S)^{-\alpha}\Vert f\Vert_{L^p(S)})
\end{equation} 
\end{proposition}

\begin{proof}
By Lemma \ref{lemmaSobolevPoincareIneguality} and by statement \eqref{item2} from Lemma \ref{lemma:ProjectProperties}
\[
\begin{split}
\bigg(\kint_{Q_S}|f|^q\,dH^s\bigg)^{1/q}\le &\bigg(\kint_{Q_S}|f-P_Qf|^q\,dH^s\bigg)^{1/q}+\bigg(\kint_{Q_S}|P_Qf|^q\,dH^s\bigg)^{1/q}
\\
\le &\,c\bigg[r^\alpha \bigg(\kint_{2Q_S}(f^\sharp_{\alpha,S})^p\,dH^s\bigg)^{1/p}+\kint_{Q_S}|f|\,dH^s\bigg].
\end{split}
\]

Choosing $Q=Q(x,\diam S)$, where $x$ is any point in $S$, and using the $s$-regularity of $S$, we get \eqref{SobolevEmbedding}.
\end{proof}

For the first order Sobolev space, the definition in terms of $L^p$-properties of sharp maximal functions makes sense also 
in general situation of a metric measure space \cite{HajlaszKinnunen}. So far, very little is known about the higher order case. 
The problem is that in the definition of the sharp maximal function, we need a family of polynomials with special properties. 
In case of $s$-sets in $\R^n$, $n-1<s\le n$,
such families do exist, see Section~\ref{SecSharpMaximalFunctions}. 
In more general situation, one could use the related technique assuming that some  polynomial type functions exist. 
This kind of an approach is used for example in \cite{LiuLuWheeden},
where a version of polynomials on metric spaces equipped with a doubling measure have been proposed.
See also \cite{AlabernMateuVerdera}, where a characterization of higher order Sobolev spaces via a quadratic multiscale expression is used to
propose an other definition for Sobolev functions on any metric space.

%\nocite{CoifmanWeiss}
%\bibliographystyle{abbrv}
%\bibliography{references}

\bigskip
\noindent Addresses:

\smallskip
\noindent L.I.: Department of Mathematics, P.O. Box 11100, 
FI-00076 Aalto University, Finland. \\
\noindent 
E-mail: {\tt lizaveta.ihnatsyeva@aalto.fi}

\noindent R.K.: Department of Mathematics and Statistics,
P.O. Box 68, FI-00014 University of Helsinki, Finland. \\
\noindent 
E-mail: {\tt riikka.korte@helsinki.fi}\\

\end{document}